\documentclass{birkjour}
 \newtheorem{theorem}{Theorem}[section]
 \newtheorem{corollary}[theorem]{Corollary}
 \newtheorem{lemma}[theorem]{Lemma}
 
 \theoremstyle{definition}
 
 \newtheorem{remark}[theorem]{Remark}
 
 \theoremstyle{remark}
 \newtheorem*{ack}{Acknowledgements}
 \numberwithin{equation}{section}
\usepackage{amssymb,amsmath,amsthm,amsfonts}

\usepackage{verbatim} 

\newcommand\R{\mathbb{R}}

\newcommand\E{\mathbb{E}}

\renewcommand\P{\mathbb{P}}

\newcommand\D{{\mathcal D}}
\renewcommand{\leq}{\leqslant}
\renewcommand{\geq}{\geqslant}
\renewcommand{\le}{\leqslant}
\renewcommand{\ge}{\geqslant}
\renewcommand{\Re}{\operatorname{Re}}

\newcommand{\sgn}{\operatorname{sgn}}
\newcommand{\Var}{\operatorname{Var}}

\newcommand\eps{{\varepsilon}}

\begin{document}

%
%
%
%
%
%
%
%
%

\title[An Inequality Related to Negative Definite Functions]{A Probabilistic Inequality Related to\\ Negative Definite Functions}

\author[M.~Lifshits]{Mikhail Lifshits}

\address{%
St.\ Petersburg State University,\br
Department of Mathematics and Mechanics,\br
198504 Stary Peterhof,\br
Bibliotechnaya pl.~2,\br
Russia}
\email{lifts@mail.rcom.ru}

\author[R.L.~Schilling]{Ren\'e L.\ Schilling}
\address{Institute of Mathematical Stochastics,\br
TU Dresden,\br
D-01062 Dresden,\br
Germany}
\email{rene.schilling@tu-dresden.de}

\author[I.~Tyurin]{Ilya Tyurin}
\address{Moscow State University,\br
Department of Mechanics and Mathematics,\br
Leninskie gory~1,\br
119991 Moscow,\br
Russia}
\email{itiurin@gmail.com}



\date{}

\begin{abstract}
    We prove that for any pair of i.i.d.\ random vectors  $X, Y$ in $\R^n$ and
    any real-valued continuous negative definite function $\psi:\R^n\to \R$ the inequality
    $$
        \E\, \psi(X-Y) \leq \E\, \psi(X+Y).
    $$
    holds. In particular, for $\alpha \in (0,2]$ and the Euclidean norm $\|\cdot\|_2$ one has
    $$
        \E \|X-Y\|_2^\alpha \leq \E \|X+Y\|_2^\alpha.
    $$
The latter inequality is due to A.\ Buja et al.\ \cite{BLRS1994} where it is used for some
applications in multivariate statistics. We show a surprising connection with bifractional
Brownian motion and provide some related counter-examples.
\end{abstract}

\maketitle

\section{Introduction}

Let $X,Y$ be i.i.d.\ random variables with finite expectations. Then one has
\begin{equation} \label{e1}
  \E |X-Y| \leq \E |X+Y|.
\end{equation}
The inequality \eqref{e1} appeared recently in an analytic context (properties of
integrable functions) \cite{XXX}. Since \eqref{e1} is a nice fact in itself and
since it seems not to be well known in the probabilistic community, it is desirable
to search for adequate proofs and to explore possible extensions of it. For instance,
for which values of $\alpha$ do we have
\begin{equation} \label{e2}
  \E |X-Y|^\alpha \leq \E |X+Y|^\alpha\, ?
\end{equation}
As before, we assume that $X$ and $Y$ are i.i.d.\ and $\E |X|^\alpha<\infty$.

Proving  \eqref{e1} is a non-trivial exercise for a probability course. If $X,Y$
are real-valued, one way to see this inequality is to use the identity
\[
   \E |X+Y| - \E |X-Y| =2 \int_0^\infty \left[\P(X>r)-\P(X<-r) \right]^2 dr.
\]
For \eqref{e2} we are, however, not aware of a similar elementary approach.
On the other hand, A.\ Buja et al.\ prove in \cite{BLRS1994} even a multivariate version of
\eqref{e2}: for any pair of i.i.d.\ random vectors  $X, Y$ in $\R^n$, any
$\alpha \in (0,2]$ and for a class of norms $\|\cdot\|$ on $\R^n$ including the
Euclidean norm $\|\cdot\|_2$ the estimate
\begin{equation} \label{e2v}
     \E \|X-Y\|^\alpha \leq \E \|X+Y\|^\alpha
\end{equation}
holds true. The elegance of this inequality is obvious; at the same time we stress
that it arises from statistical applications. In any case it merits to be better
known in the probabilistic community!

In Section~\ref{s:rene} we give an extension of \eqref{e2v} by replacing the norm
with an arbitrary negative definite function. Moreover, we show how this fact
extends to an arbitrary number of i.i.d.\ random vectors.
In Sections \ref{s:gaussian} and  \ref{s:bifrac} we establish a surprising connection to some recent
advances in the theory of random processes related to  bifractional Brownian motion.
A counterexample to \eqref{e2} with $\alpha\in (2,\infty)$ is given in
Section~\ref{s:counter}.

\section{Main result}\label{s:rene}

Consider the class of \emph{continuous real-valued negative definite functions}, i.e.\
characteristic exponents of symmetric L\'evy processes. The notion of negative definite
function goes back to Schoenberg; good sources are the books \cite{BF1975} and
\cite{SSV2011}. Recall that a continuous real-valued negative definite function is
uniquely given by its L\'evy-Khintchine representation
\begin{equation}\label{LK_repr}
   \psi(\xi)=a+\frac12\,\langle Q\xi,\xi\rangle +\int_{\R^n\setminus\{0\}}
   \left(1-\cos\langle \xi,u\rangle\right)\nu(du),
   \qquad \xi\in \R^n,
\end{equation}
where $a\geq 0$ is a constant, $Q\in\R^{n\times n}$ is a symmetric positive semidefinite
matrix and $\nu$ is the L\'evy measure, i.e.\ a measure on $\R^n\setminus\{0\}$ satisfying
the integrability condition
\begin{equation}\label{LK_mes}
    \int_{\R^n\setminus\{0\}}  \min\{\|u\|^2_2,1\} \nu(du)<\infty.
\end{equation}
Without loss of generality, we will always assume that $a=0$, i.e.\ $\psi(0)=0$. 
For our discussion it is worth noticing that $(\xi,\eta)\mapsto \sqrt{\psi(\xi-\eta)}$ is
always a metric. A deep theorem of Schoenberg states that a metric space $(\R^n ,d)$ can
be isometrically embedded into an (in general infinite-dimensional) Hilbert space
$\mathcal H$ if, and only if, $d(\xi,\eta)$ is of the form $d_\psi(\xi,\eta)=\sqrt{\psi(\xi-\eta)}$,
cf.\ \cite{schoenberg38}, \cite[p.\ 187]{BL2000} as well as \cite{jac-et-al} for a discussion
of metric measure spaces related to the metric $d_\psi$.

An important subclass of continuous negative definite functions are the spherically symmetric
negative definite functions. These are of the form
\begin{equation}\label{bernstein}
   \xi\mapsto f(\|\xi\|_2^2) \qquad \text{where $f$ is a Bernstein function}.
\end{equation}

Recall that a \emph{Bernstein function} is a function $f:\R_+\to\R_+$ which admits the
following L\'evy-Khintchine representation
\[
     f(\lambda)= a+b\lambda+\int_0^\infty \left(1-e^{-t\lambda} \right) \mu(dt);
\]
here $a,b\ge 0$ are constants and $\mu$ is a measure on $(0,\infty)$ satisfying the
integrability condition $\int_0^\infty \min\{t,1\}\,\mu(dt)<\infty$. In probability theory
Bernstein functions arise as the characteristic exponents of the Laplace transform of
subordinators, i.e.\ increasing one-dimensional L\'evy processes. Bernstein functions,
many examples and their connections to various fields of mathematics are discussed in the
monograph \cite{SSV2011}. It is easy to see that Bernstein functions are infinitely many times
differentiable, increasing, concave; moreover, they grow at most linearly. Typical examples
are $\lambda\mapsto\log(1+\lambda)$ and $\lambda\mapsto f_\beta(\lambda):=\lambda^\beta$ for
$0<\beta\leq 1$. Note that the composition $f\circ\psi$ of a Bernstein function $f$ with a
continuous real-valued negative definite function $\psi$ is again a continuous real-valued
negative definite function. At the level of stochastic processes this corresponds to
\emph{Bochner's subordination} of the L\'evy process with characteristic exponent $\psi$ by
the subordinator with the Laplace exponent $f$.

Using the Bernstein functions $f_\beta$ with $\beta = \alpha/2$ and $0<\alpha\leq 2$ we
obtain
\begin{align*}
    \xi \mapsto \|\xi\|_2^\alpha = f_{\alpha/2}(\|\xi\|^2),
    \qquad
    0< \alpha\le 2,\\
    \xi \mapsto d_\psi(\xi,0)^\alpha = \sqrt{\psi(\xi)}^{\,\alpha} = f_{\alpha/2}(\psi(\xi)),
    \qquad
    0< \alpha\le 2,
\end{align*}
as examples for real-valued continuous negative definite functions. Note that the functions
defined by \eqref{bernstein} are characteristic exponents of subordinate Brownian motions.

We prove the following result extending \eqref{e2v}.

\begin{theorem}\label{t:main}
Let $\psi$ be a real-valued
continuous negative definite function on $\R^n$. For any pair of i.i.d.\ random
vectors $X,Y$ in $\R^n$ it is true that
\begin{equation} \label{eq:psi}
     \E\, \psi(X-Y) \leq \E\, \psi(X+Y).
\end{equation}
\end{theorem}

\begin{proof}
Without loss of generality we may assume that $a=0$ and $Q=0$ -- in both cases the inequality
\eqref{eq:psi} is elementary.

Using the L\'evy-Khintchine
representation of $\psi$ we get
\begin{align*}
   \E\, \psi(X+Y)
   &= \E \int_{\R^n\setminus\{0\}}\big(1-\cos\langle X+Y,u\rangle\big)\,\nu(du)
\\
   &= \E \int_{\R^n\setminus\{0\}}\big(1-\Re\exp(i\langle X+Y,u\rangle)\big)\,\nu(du)
\\
   &= \int_{\R^n\setminus\{0\}}\big(1-\Re\, \E\exp(i\langle X+Y,u\rangle)\big)\,\nu(du)
\\
   &= \int_{\R^n\setminus\{0\}}\left(1-\Re\left[ \E\exp(i\langle X,u\rangle)\right]^2\right)\nu(du).
\end{align*}
A similar calculation yields
\begin{align*}
   \E\, \psi(X-Y)
   &= \E \int_{\R^n\setminus\{0\}}\big(1-\cos\langle X-Y,u\rangle\big)\,\nu(du)
\\
   &= \E \int_{\R^n\setminus\{0\}}\big(1-\Re\exp(i\langle X-Y,u\rangle)\big)\,\nu(du)
\\
   &= \int_{\R^n\setminus\{0\}}\big(1-\Re\, \E\exp(i\langle X-Y,u\rangle)\big)\,\nu(du)
\\
   &= \int_{\R^n\setminus\{0\}} \left(1- \left| \E\exp(i\langle X,u\rangle)\right|^2 \right)\nu(du).
\end{align*}

Using the elementary estimate $\Re(z^2)\le |z^2|= |z|^2$ we obtain \eqref{eq:psi}.
\end{proof}

\begin{remark}
Let $X_1,\dots X_{2m}$ be
i.i.d.\ random variables in $\R^n$ and $\eps_j=\pm 1$ (non-random, or even
random but independent of the $X_1, \dots, X_{2m}$) constants satisfying
$\sum_{j=1}^{2m}\eps_j=0$. Then
\begin{equation} \label{evpsim}
  \E\, \psi\left(\sum_{j=1}^{2m} \eps_jX_j\right) \leq
  \E\, \psi\left(\sum_{j=1}^{2m}
  X_j\right) .
\end{equation}
This follows if we use Theorem~\ref{t:main} for $X = \sum_{j=1}^{2m} \eps_j^+ X_j$ and
$Y = \sum_{j=1}^{2m} \eps_j^- X_j$.
\end{remark}

Using the distance function $d_\psi(\xi,\eta):= \sqrt{\psi(\xi-\eta)}$
related to a real-valued continuous negative definite function $\psi$ we get the
following counterpart of \eqref{e2v}.
\begin{corollary}\label{c:main}
Let $\psi:\R^n\to\R$ be a real-valued
continuous negative definite function,
$d_\psi(\xi,\eta)= \sqrt{\psi(\xi-\eta)}$ the associated metric and
$0<\alpha\leq 2$. For any pair of i.i.d.\ random vectors $X,Y$ in $\R^n$ it is true that
\begin{equation} \label{eq:distance}
  \E\, d_\psi^\alpha(X-Y) \leq \E\, d_\psi^\alpha(X+Y).
\end{equation}
\end{corollary}

\begin{remark}\label{r:converse-conjecture}
Assume that $\psi:\R^n \to \R$ is a continuous function such that $\psi(0)=0$ and $\psi(\xi)=\psi(-\xi)$. 
If \eqref{eq:psi} holds for this $\psi$ and \emph{any} random variable $X$ (and an independent 
copy $Y$ of $X$), then one can show that the kernel $K_\psi(\xi,\eta):=\psi(\xi+\eta)-\psi(\xi-\eta)$ 
is positive definite. We 
wonder whether this already entails that $\psi$ is a continuous negative definite function.
\end{remark}

\section{A relation to random processes} \label{s:gaussian}

We will show now that
the inequality \eqref{eq:psi} has an interesting relation to Gaussian processes.
Let $\psi:\R^n\to\R$ be a real-valued continuous negative definite function defined on $\R^n$.

\begin{lemma}\label{l:pos-def}
    The kernel $K^\psi(\xi,\eta)=\psi(\xi+\eta)-\psi(\xi-\eta)$ is positive definite.
\end{lemma}
\begin{proof}
By the L\'evy-Khintchine formula \eqref{LK_repr} we get
$$
        K^\psi(\xi,\eta)
        = 2 \langle Q\xi,\eta\rangle  + \int_{\R^n\setminus \{0\}}
          \left(\cos( \langle\xi-\eta,u\rangle) -  \cos( \langle\xi+\eta,u\rangle)  \right)\,\nu(du).
$$
Using the elementary trigonometric identity
\[
        \cos\langle\xi-\eta,u\rangle -  \cos\langle\xi+\eta,u\rangle
        = 2\sin \langle\xi,u\rangle \sin \langle\eta,u\rangle,
\]
we see that
\[
    K^\psi(\xi,\eta)
        = 2 \langle Q\xi,\eta\rangle  + 2 \int_{\R^n\setminus \{0\}}
             \sin \langle\xi,u\rangle \sin \langle\eta,u\rangle  \,\nu(du).
\]
Now let $S$ be a finite set and $(\lambda_\xi,\: \xi\in S)$ be complex numbers. Then
\begin{align*}
    &\sum_{\xi,\eta\in S} K^\psi(\xi,\eta)\lambda_\xi\overline{\lambda}_\eta\\
        &\quad= 2 \sum_{\xi,\eta\in S} \lambda_\xi \overline{\lambda_\eta} \langle Q\xi,\eta\rangle
           + 2\int_{\R^n\setminus \{0\}}
            \left( \sum_{\xi,\eta\in S} \lambda_\xi \sin \langle\xi,u\rangle
            \:\overline{\lambda_\eta \sin \langle\eta,u\rangle} \right)\nu(du)\\
        &\quad= 2\,\left\langle Q \sum_{\xi\in S} \lambda_\xi\, \xi,    \sum_{\xi\in S} \lambda_\xi\, \xi
          \right\rangle
          + 2\int_{\R^n\setminus \{0\}}
             \left| \sum_{\xi\in S} \lambda_\xi\sin \langle\xi,u\rangle  \right|^2\nu(du)\\
        &\quad\ge 0,
    \end{align*}
which means that $K^\psi(\cdot,\cdot)$ is positive definite.
\end{proof}

\begin{remark}
  A special case of Lemma \ref{l:pos-def} for powers of $\ell_p$-norms is proved in
  \cite{BLRS1994}.
\end{remark}

\begin{proof}[Probabilistic proof of Theorem~\ref{t:main}]
Since $K^\psi(\xi,\eta)$ is positive definite, there is a
centred Gaussian process
$\big(G^\psi_\xi,\; \xi\in\R^n\big)$ whose covariance function is $K^\psi(\xi,\eta)$.

For given i.i.d.\ random vectors $X,Y\in \R^n$ set
$$
    Z^\psi := \int_{\R^n} G^\psi_{\xi} \, P(d\xi),
$$
where $P$ stands for the common distribution of $X$ and $Y$. Then
\begin{align*}
  0
  \leq \Var(Z^\psi)
  &=
  \int_{\R^n} \int_{\R^n}  K^\psi(\xi,\eta) \, P(d\xi)\, P(d\eta)
\\
  &=\int_{\R^n} \int_{\R^n}  \left(\psi(\xi+\eta) - \psi(\xi-\eta)\right)\, P(d\xi)\, P(d\eta)
\\
  &= \E\, \psi(X+Y) - \E\, \psi(X-Y),
\end{align*}
and we obtain again
$
  \E\, \psi(X-Y) \le \E\, \psi(X+Y).
$
\end{proof}

\section{Relation to bifractional Brownian motion} \label{s:bifrac}

In some most important cases it is possible to identify the Gaussian process
$\big(G^\psi_\xi, \; \xi\in\R^n\big)$ of Section~\ref{s:gaussian} with
 \emph{bifractional Brownian motion} (bBm). The latter
process was introduced by Houdr\'e and Villa in \cite{HV2003} as a centred Gaussian process
$B^{H,K} = \big(B^{H,K}_t,\; t\in \R^n\big)$ with covariance function
$$
    R^{H,K}(t,s)
    := \E\left(B^{H,K}_t B^{H,K}_s\right)
    = 2^{-K}\left((||t||_2^{2H}+||s||_2^{2H})^K - ||t-s||_2^{2HK}\right),
$$
where $s,t\in \R^n$.
For  $n=1$, $K=1$ we get the usual fractional Brownian motion $B^H$ with Hurst index $H$. Originally,
the process was defined for the parameters $H\in(0,1]$ and ${K\in(0,1]}$. Bardina and Es-Sebaiy
\cite{BE2010} recently proved that  $B^{H,K}$ exists for all $(H,K)\in \D$, where
\[
    \D:= \{H,K: 0<H\le 1, 0<K\le 2, H\cdot K \le 1\}.
\]
(The possibility of such an extension was already indicated in the earlier work
by Lei and Nualart \cite{LN2009} who established an integral representation
relating $B^{H,K}$ with fractional Brownian motion $B^{HK}$).

For $\psi(\xi) := |\xi|^\alpha$, $0<\alpha\le 2$, and
\[
   G^\psi_\xi := 2^{\alpha/2}\,  \sgn(\xi)\, B^{\frac 12, \alpha}_{|\xi|}, \qquad \xi\in\R,
\]
it is trivial to see that
\[
   \E \left(G^\psi_\xi G^\psi_\eta\right) =  \sgn(\xi\eta) \, 2^{\alpha}\,
   \E\left( B^{\frac 12, \alpha}_{|\xi|}, B^{\frac 12, \alpha}_{|\xi|} \right)
   =   |\xi+\eta|^\alpha - |\xi-\eta|^\alpha =K^\psi(\xi,\eta).
\]
Therefore,
we are led to a probabilistic interpretation of the inequality \eqref{e2} through $B^{\frac 12, \alpha}$.

\begin{remark}
In higher dimensions bi-fractional Brownian motion does not show up in the context of our inequalities 
(nor do we rely on bBm with $H\neq\tfrac 12$); therefore it becomes natural to search for the extensions 
of bBm based upon general negative definite functions. This will be done elsewhere.
\end{remark}

\section{A counterexample}\label{s:counter}

The inequality \eqref{e2} trivially extends to the case $\alpha=\infty$
in the following sense. Let
\begin{align*}
  M &= \sup\{r: \P(X<r)<1 \} = \mathop{\mathrm{ess\,sup}} X;  \\
  m &=  \sup\{r: \P(X<r)=0 \} = \mathop{\mathrm{ess\,inf}} X.
\end{align*}
Then
\[
   \|X-Y\|_\infty
   = M-m \le 2 \max\{|M|, |m| \}
   = \|X+Y\|_\infty.
\]

Without further assumptions the inequality \eqref{e2} will, in general,
not hold, for $2<\alpha<\infty$. To see this, fix $\alpha\in(2,\infty)$
and $c>0$. For any $M\ge c$ set $q:=c/M$ and $p:=1-q$. Let $X_M,Y_M$ be
i.i.d.\ random variables such that
\begin{gather*}
  \P(X_M=1) = \P(Y_M=1)=p; \\
  \P(X_M=-M) =\P(Y_M=-M)=q.
\end{gather*}
If $M\ge 1$, then
\begin{align*}
\E |X_M-& Y_M|^\alpha - \E |X_M+Y_M|^\alpha
\\
    &= 2pq\left[ (M+1)^\alpha-(M-1)^\alpha \right] -2^\alpha M^\alpha
    q^2 -2^\alpha p^2
\\
    &\ge  4pq \alpha M^{\alpha-1} -2^\alpha M^\alpha q^2-2^\alpha p^2
\\
    &=  M^{\alpha-2}(4p\alpha c -2^\alpha c^2) -2^\alpha p^2.
\end{align*}
Hence, whenever $c<2^{2-\alpha} \alpha$ and $M$ is large enough,
\[
  \E |X_M-Y_M|^\alpha - \E |X_M+Y_M|^\alpha  >0,
\]
and \eqref{e2} fails.

\begin{remark}
   Further counterexamples are presented in \cite{BLRS1994}.
\end{remark}

\begin{ack}
The authors are grateful to H.~Kempka, A.~Koldobskii, W.~Linde and A.~Nekvinda
for pointing out the problem and for valuable discussions. Special thanks are
due to an anonymous referee of an earlier version of this note for pointing out the
crucial reference \cite{BLRS1994}.

The research of Russian authors was supported by the RFBR-DFG grant 09-01-91331,
RFBR grants 10-01-00154$a$, 10-01-00397$a$, and the Federal Focused Programme
2010-1.1-111-128-033.
\end{ack}

\end{document}